\documentclass[a4paper]{amsart}
\usepackage{amssymb,amsmath,amsthm,hyperref}

\theoremstyle{plain}
\newtheorem{thm}{Theorem}[section]
\newtheorem{theorem}[thm]{Theorem}
\newtheorem{lemma}[thm]{Lemma}
\newtheorem{prop}[thm]{Proposition}

\theoremstyle{definition}
\newtheorem{defn}[thm]{Definition}

\newcommand{\Reals}{\mathbb{R}}

\newcommand{\Rats}{\mathbb{Q}}

\newcommand{\Rexp}{\Reals_{\exp}}
\newcommand{\ecl}{\operatorname{ecl}}

\newcommand{\tuple}[1]{\overline{#1}}

\newcommand{\ktup}{\tuple{k}}
\newcommand{\ltup}{\tuple{l}}

\newcommand{\xtup}{\tuple{x}}
\newcommand{\ytup}{\tuple{y}}
\newcommand{\ztup}{\tuple{z}}

\newcommand{\pow}{\lambda}  
\newcommand{\ptup}{\tuple{\pow}}
\newcommand{\Xtup}{\tuple{X}}

\renewcommand{\le}{\leqslant}
\renewcommand{\ge}{\geqslant}

\newcommand{\ld}{\operatorname{ldim}}
\newcommand{\trd}{\operatorname{td}}

\newcommand{\linearlydisjoint}{\bot}

\newcommand{\spanofover}[2]{\left<#1\right>_{#2}}

\title{A Schanuel property for exponentially transcendental powers}

\author{Martin Bays}
\author{Jonathan Kirby}
\author{A.J.~Wilkie}

\begin{document}

\begin{abstract}
 We prove the analogue of Schanuel's conjecture for raising to the power of an exponentially transcendental real number. All but countably many real numbers are exponentially transcendental. We also give a more general result for several powers in a context which encompasses the complex case.
\end{abstract}

\maketitle

\section{Introduction}

We prove a Schanuel property for raising to a real power:
\begin{thm}\label{Real powers}
  Let $\pow\in\Reals$ be exponentially transcendental, let $\ytup\in(\Reals_{>0})^n$,
  and suppose $\ytup$ is multiplicatively independent. Then
  \[ \trd(\ytup,\ytup^\pow/\pow) \ge n. \]
\end{thm}
Here and later, $\trd(X/Y)$ denotes the transcendence degree of the field extension $\Rats(X,Y)/\Rats(Y)$ (for $X$, $Y$ subsets of the ambient field, in this case $\Reals$). To say that $\ytup$ is multiplicatively independent means that if $m_1,\ldots,m_n \in \mathbb{Z}$ and $\prod y_i^{m_i} = 1$ then $m_i = 0$ for each $i$. The usual exponential function $\exp:\Reals \to \Reals$ makes the reals into an \emph{exponential field}, formally a field of characteristic zero equipped with a homomorphism from its additive to multiplicative groups. In any exponential field $\langle F;+,\cdot,\exp\rangle$, we say that an element $x \in F$ is \emph{exponentially algebraic} in $F$ iff there is $n \in \mathbb{N}$, $\xtup = (x_1,\ldots,x_n) \in F^n$, and exponential polynomials $f_1,\ldots,f_n \in \mathbb{Z}[\Xtup,e^{\Xtup}]$ such that $x = x_1$, $f_i(\xtup,e^{\xtup}) = 0$ for each $i=1,\ldots,n$, and the determinant of the Jacobian matrix 
\[\begin{pmatrix}
\frac{\partial f_1}{\partial X_1} & \cdots & \frac{\partial f_1}{\partial X_n} \\
\vdots &\ddots & \vdots \\
\frac{\partial f_n}{\partial X_1} & \cdots & \frac{\partial f_n}{\partial X_n} \\
\end{pmatrix}\]
is nonzero at $\xtup$. If $x$ is not exponentially algebraic in $F$ we say it is \emph{exponentially transcendental} in $F$. More generally, for a subset $A$ of $F$, we can define the notion of $x$ being \emph{exponentially algebraic over $A$} with the same definition except that the $f_i$ can have coefficients from $A$. Observe that the non-vanishing of the Jacobian in the reals means that $\xtup$ is an isolated zero of the system of equations, and hence all but countably many real numbers are exponentially transcendental. Thus a consequence of theorem~\ref{Real powers} is that the numbers $\pow, \pow^\pow, \pow^{\pow^2},\pow^{\pow^3},\ldots$ are algebraically independent for all but countably many $\pow$, although, unfortunately, one does not know any explicit $\pow$ for which this is true.

This paper contains a complete proof of theorem~\ref{Real powers}, assuming only some knowlege of o-minimality from the reader (and using a theorem of Ax). The paper \cite{EAEF} of the second author develops the theory of exponential algebraicity in an arbitrary exponential field, and, using that, we can prove a more general theorem.
\begin{theorem}\label{thm:ETpower}
  Let $F$ be any exponential field, let $\pow \in F$ be exponentially transcendental, and let $\xtup \in F^n$ be such that $\exp(\xtup)$ is multiplicatively independent. Then 
\[ \trd(\exp(\xtup),\exp(\pow\xtup)/\pow) \ge n.\]
\end{theorem}
Theorem~\ref{Real powers} follows from \ref{thm:ETpower} by taking $x_i = \log y_i$.

We define the exponential algebraic closure $\ecl(A)$ of a subset $A$ of $F$ to be the set of $x \in F$ which are exponentially algebraic over $A$. In \cite{EAEF} it is shown that $\ecl$ is a pregeometry in any exponential field, and hence we have notions of dimension and independence. We also prove a general Schanuel property for raising to several independent powers, which uses a slightly subtle notion of relative linear dimension. For any subfield $K$ of $F$, we can think of $F$ as a $K$-vector space.
For subsets $X$, $Y$ of $F$, consider the $K$-linear subspaces $\spanofover{XY}{K}$ and $\spanofover{Y}{K}$ of $F$ generated by $X \cup Y$ and $Y$ respectively. We define $\ld_K(X/Y)$ to be the $K$-linear dimension of the quotient $K$-vector space $\spanofover{XY}{K}/\spanofover{Y}{K}$.
\begin{thm}\label{thm:PowersSC}
  Let $F$ be any exponential field, let $\ker$ be the kernel of its exponential map, let $C$ be an $\ecl$-closed subfield of $F$, and let $\ptup$ be an $m$-tuple which is exponentially algebraically independent over $C$. Then for any tuple $\ztup$ from $F$:
  \[\trd(\exp(\ztup)/C,\ptup) + \ld_{\Rats(\ptup)}(\ztup/\ker) -
  \ld_{\Rats}(\ztup/\ker) \ge 0 .\]
\end{thm}
The reader who is interested only in the real case may ignore all the references to \cite{EAEF}. On the other hand, the reader who is unfamiliar with o-minimality may prefer to ignore that part of this paper and instead refer to the algebraic proof of proposition~\ref{fact:genSelfSuff} in \cite{EAEF}.

\section{A Schanuel property for exponentiation}

We need the following relative Schanuel property for exponentiation itself.
\begin{prop}\label{fact:genSelfSuff}
  Let $F$ be an exponential field and let $\ptup\in F^m$ be exponentially algebraically independent. Let $B \subseteq F$ be such that $B \cup \ptup$ is a basis for $F$ with respect to the pregeometry $\ecl$. Let $C = \ecl(B)$. Then for any $\ztup\in F^n$,
  \[ \trd(\ptup,\ztup,\exp(\ptup),\exp(\ztup)/C) - \ld_\Rats(\ptup,\ztup/C) \ge m.\]
\end{prop}
\begin{proof}
  Theorem~1.2 of \cite{EAEF} states that 
  $\trd(\ptup,\ztup,\exp(\ptup),\exp(\ztup)/C) - \ld_\Rats(\ptup,\ztup/C)$ is at least the dimension of the $(m+n)$-tuple $(\ptup,\ztup)$ over $C$ with respect to the pregeometry $\ecl$. Since $\ptup$ is $\ecl$-independent over $C$ by assumption, this dimension is at least $m$.
\end{proof}

We give a more direct proof of proposition~\ref{fact:genSelfSuff} in the real case. Firstly, by theorem~4.2 of \cite{JW98}, a real number $x$ is in the exponential algebraic closure $\ecl(A)$ of a subset $A$ of $\Reals$ iff it lies in the definable closure of $A$ in the structure $\Rexp = \langle\Reals;+,\cdot,\exp\rangle$. Definable closure is always a pregeometry in an o-minimal field, so $\ecl$ is a pregeometry on $\Rexp$.

For each $i = 1,\ldots,m$, let $K_i = \ecl(B \cup \ptup \smallsetminus \pow_i)$, so $C = \bigcap_{i=1}^m K_i$. Then for each $i$, $\pow_i \notin K_i$, but for each $a \in \Reals$ there is a function $\theta:\Reals \to \Reals$, definable in $\Rexp$ with parameters from $K_i$, such that $\theta(\pow_i)=a$. By o-minimality of $\Rexp$, $\theta$ is differentiable at all but finitely many $x \in \Reals$, and hence this exceptional set is contained in $K_i$. Thus $\theta$ is differentiable on an open interval containing $\pow_i$. Suppose that $\psi:\Reals \to \Reals$ is another such function with $\psi(\pow_i)=a$. Again by o-minimality, the boundary of the set $\{x \in \Reals \mid\psi(x) = \theta(x)\}$ is finite and contained in $K_i$, so $\theta$ and $\psi$ agree on an open interval containing $\pow_i$. It follows that there is a well-defined function $\partial_i:\Reals \to \Reals$ which sends $a$ to $\frac{d \theta}{dx}(\pow_i)$, where $\theta$ is any function definable in $\Rexp$ with parameters from $K_i$ such that $\theta(\pow_i)=a$. It is straightforward to check that $\partial_i$ is a derivation on the field $\Reals$, with field of constants $K_i$. Furthermore, we also clearly have that $\partial_i(\exp(a)) = \partial_i(a) \exp(a)$ for any $a\in \Reals$, and that $\partial_i(p_j) = \delta_{ij}$, the Kronecker delta.

By Ax's theorem \cite[theorem~3]{Ax71}, 
$\trd(\ptup,\ztup,\exp(\ptup),\exp(\ztup)/C) - \ld_\Rats(\ptup,\ztup/C)$ is at least the rank of the matrix
\[\begin{pmatrix}\partial_1 z_1 & \cdots & \partial_1 z_n & \partial_1 \pow_1 & \cdots & \partial_1 \pow_m\\
\vdots & & \vdots & \vdots & & \vdots\\ 
\partial_m z_1 & \cdots & \partial_m z_n & \partial_m \pow_1 & \cdots & \partial_m \pow_m\end{pmatrix}\]
which is $m$ since the right half is just the $m \times m$ identity matrix. That completes the proof of proposition~\ref{fact:genSelfSuff} in the real case. The general case works the same way, but a different and much more involved argument is used in \cite{EAEF} to produce the derivations $\partial_i$ without using o-minimality.

\section{Linear disjointness}

The other key ingredient in the proofs is the concept of linear disjointness. We briefly recall the definition and some basic properties.

\begin{defn} Let $F$ be a field, and let $K$, $L$, and $E$ be subfields of $F$ with $E \subseteq K \cap L$. Then $K$ is \emph{linearly disjoint from $L$ over $E$}, written $K\linearlydisjoint_E L$, iff every tuple $\ktup$ of elements of $K$ that is $E$-linearly independent is also $L$-linearly independent.
\end{defn}

\begin{lemma}\label{lemma:LDElem}  \mbox{ }
  \begin{enumerate}
    \item[(i)] $K\linearlydisjoint_E L$ iff
      $L\linearlydisjoint_E K$
    \item[(ii)] $K\linearlydisjoint_E L$ iff for any tuple
      $\ltup$ from $L$, $\ld_K(\ltup) = \ld_E(\ltup)$
    \item[(iii)] If $\ktup$ is algebraically independent
      over $L$, then $E(\ktup)\linearlydisjoint_E L$.
  \end{enumerate}
\end{lemma}
\begin{proof}
      (i) and (ii) are straightforward; (iii) is proposition~VIII~3.3 of \cite{Lang93}.
\end{proof}

\begin{lemma}\label{lemma:LDIneq}
  Suppose $K\linearlydisjoint_E L$. Then for any tuple $\xtup$ from $F$ and any subset $A\subseteq L$,
  \[ \ld_K(\xtup/L) - \ld_E(\xtup/L) \le \ld_K(\xtup/A) - \ld_E(\xtup/A). \]
\end{lemma}
\begin{proof}
  Let $\ltup \in L$ be a finite tuple such that $\ld_K(\xtup/\ltup A)=\ld_K(\xtup/L)$ and
  $\ld_E(\xtup/\ltup A)=\ld_E(\xtup/L)$.

  Now:
  \begin{align*}
    \ld_K(\xtup/A) - \ld_K(\xtup/\ltup A)
    &= \ld_K(\ltup/A) - \ld_K(\ltup/\xtup A)  &\textrm{(by the addition formula)} \\
    &= \ld_E(\ltup/A) - \ld_K(\ltup/\xtup A)  &\textrm{(by Lemma \ref{lemma:LDElem}(ii))} \\
    &\ge \ld_E(\ltup/A) - \ld_E(\ltup/\xtup A) \\
    &= \ld_E(\xtup/A) - \ld_E(\xtup/\ltup A)  &\textrm{(by the addition formula).} \\
  \end{align*}
\end{proof}

\section{Proofs of the main theorems}

\begin{proof}[Proof of theorem~\ref{thm:PowersSC}]
  By proposition~\ref{fact:genSelfSuff}, for any tuple $\ztup$ from $F$ we have:
\[ \trd(\ztup,\exp(\ztup),\ptup,\exp(\ptup)/C) - \ld_{\Rats}(\ztup,\ptup/C) \ge m.\]
Expanding using the addition formula gives
\begin{multline*}
\trd(\ptup/C) + \trd(\ztup/C,\ptup) + \trd(\exp(\ztup)/C,\ptup,\ztup)\\
    + \trd(\exp(\ptup)/C,\ptup,\ztup,\exp(\ztup)) 
    - \ld_\Rats(\ptup/C,\ztup) - \ld_\Rats(\ztup/C) \ge m.
\end{multline*}
Since $\ptup$ is algebraically independent over $C$, we have $\trd(\ptup/C) = m$, and we deduce
\begin{multline}\label{e1}
 \trd(\ztup/C,\ptup) + \trd(\exp(\ztup)/C,\ptup) + \trd(\exp(\ptup)/C,\exp(\ztup)) \\
    - \ld_\Rats(\ptup/C,\ztup) - \ld_\Rats(\ztup/C) \ge 0.
\end{multline}

We also have:
  \begin{equation}\label{e2}
   \trd(\exp(\ptup)/C,\exp(\ztup)) \le \ld_{\Rats}(\ptup/C,\ztup)
  \end{equation}
because if $\pow_1,\ldots,\pow_t$ form a $\Rats$-linear basis for $\ptup$ over $(C,\ztup)$, then for $i>t$, $\exp(\pow_i)$ is in the algebraic closure of $(C,\exp(\ztup),\exp(\pow_1),\ldots,\exp(\pow_t))$. A similar argument shows
\begin{equation}\label{e3}
\trd(\ztup/C,\ptup) \le \ld_{\Rats(\ptup)}(\ztup/C)
\end{equation}
since if $z_i$ is in the $\Rats(\ptup)$-linear span of $(z_1,\ldots,z_t,C)$ then $z_i$ is in the algebraic closure of $(C,\ptup,z_1,\ldots,z_t)$.

Combining (\ref{e1}) with (\ref{e2}) and (\ref{e3}) gives 
\[\trd(\exp(\ztup)/C,\ptup) + \ld_{\Rats(\ptup)}(\ztup/C) -\ld_\Rats(\ztup/C) \ge 0.\]
By lemma~\ref{lemma:LDElem}(iii), $\Rats(\ptup)$ is linearly disjoint from $C$ over $\Rats$. Also $\ker \subseteq \ecl(\emptyset) \subseteq C$, so, by lemma~\ref{lemma:LDIneq},
\[\trd(\exp(\ztup)/C,\ptup) + \ld_{\Rats(\ptup)}(\ztup/\ker) - \ld_{\Rats}(\ztup/\ker) \ge 0\]
  as required.
\end{proof}

\begin{proof}[Proof of theorem~\ref{thm:ETpower}]
  By theorem~\ref{thm:PowersSC}, taking $\ztup =  (\xtup,\pow\xtup)$,
  \begin{align*}
    \trd(\exp(\xtup),\exp(\pow\xtup)/\pow) & \ge 
    \ld_{\Rats}(\xtup,\pow\xtup/\ker) - \ld_{\Rats(\pow)}(\xtup,\pow\xtup/\ker) \\
    &= \ld_\Rats(\xtup/\ker) + \ld_\Rats(\pow\xtup/\xtup,\ker)
                - \ld_{\Rats(\pow)}(\xtup/\ker) \\
   &= n + \ld_\Rats(\pow\xtup/\xtup,\ker)
                - \ld_{\Rats(\pow)}(\xtup/\ker).
  \end{align*}

Thus it suffices to prove that $\ld_\Rats(\pow\xtup/\xtup,\ker) \ge \ld_{\Rats(\pow)}(\xtup/\ker)$.
Let $\ktup$ be a finite tuple from $\ker$ such that $\ld_\Rats(\pow\xtup/\xtup,\ker) = \ld_\Rats(\pow\xtup/\xtup,\ktup)$ and $\ld_{\Rats(\pow)}(\xtup/\ker) = \ld_{\Rats(\pow)}(\xtup/\ktup)$.

  Let $A_0 := \spanofover{\pow\xtup,\ktup}{\Rats}$. Then $\ld_{\Rats}(\pow\xtup,\ktup/\xtup,\pow^{-1}\ktup) = \ld_{\Rats}(A_0/A_0 \cap \pow^{-1}A_0)$.
  Inductively define $A_{i+1} := A_i \cap \pow^{-1}A_i$ for $i \in \mathbb{N}$. Suppose for some $i$ that $A_{i+1} = A_i$. Then multiplication by $\pow$ induces a $\Rats$-linear automorphism of $A_i$. It follows that for any $f(\pow) \in \Rats[\pow]$, multiplication by $f(\pow)$ is a $\Rats$-linear endomorphism of $A_i$. This endomorphism has trivial kernel because $f(\pow)$ is not a zero divisor of the field (unless $f(\pow) = 0$), and $A_i$ is finite-dimensional, so it is invertible. Its inverse must be multiplication by $f(\pow)^{-1}$, and hence $A_i$ is a $\Rats(\pow)$-vector space. Since $\pow$ is transcendental, $\ld_\Rats \Rats(\pow)$ is infinite, so $A_i = \{0\}$. So $\ld_\Rats A_{i+1} < \ld_\Rats A_i$ unless $A_i = \{0\}$. Thus for some $N \in \mathbb{N}$ we have $A_N = \{0\}$.

For each $i$ we have a chain of subspaces $A_{i+1} \subseteq A_{i+1} + \pow A_{i+1} \subseteq A_i$, so
\begin{align*}
  \ld_\Rats(A_i/A_{i+1}) 
  & =  \ld_\Rats(A_i/A_{i+1} + \pow A_{i+1}) + \ld_\Rats(A_{i+1} + \pow A_{i+1}/A_{i+1}) \\
  & =  \ld_\Rats(A_i/A_{i+1} + \pow A_{i+1}) + \ld_\Rats(\pow A_{i+1}/A_{i+1} \cap \pow A_{i+1}) \\
  & =  \ld_\Rats(A_i/A_{i+1} + \pow A_{i+1}) + \ld_\Rats(\pow A_{i+1}/\pow A_{i+2}) \\
  & =  \ld_\Rats(A_i/A_{i+1} + \pow A_{i+1}) + \ld_\Rats(A_{i+1}/A_{i+2}). \\
\end{align*}
  Thus inductively we obtain
\[    \ld_\Rats(A_0/A_1) = \sum_{i=0}^N \ld_\Rats(A_i/A_{i+1} + \pow A_{i+1}).\]

Now for each $i$, 
\[\ld_\Rats(A_i/A_{i+1} + \pow A_{i+1}) \ge \ld_{\Rats(\pow)}(A_i/A_{i+1} + \pow A_{i+1}) = \ld_{\Rats(\pow)}(A_i/A_{i+1})\]
hence 
\[\ld_\Rats(A_0/A_1) \ge \sum_{i=0}^N \ld_{\Rats(\pow)}(A_i/A_{i+1}) = \ld_{\Rats(\pow)}(A_0)\]
that is,
\begin{equation}\label{e7}
\ld_\Rats(\pow \xtup,\ktup/\xtup,\pow^{-1}\ktup) \ge \ld_{\Rats(\pow)}(\xtup,\pow^{-1}\ktup).
\end{equation}

But 
\begin{equation}\label{e8}
\ld_{\Rats(\pow)}(\xtup,\pow^{-1}\ktup) = \ld_{\Rats(\pow)}(\xtup,\ktup) = \ld_{\Rats(\pow)}(\xtup/\ktup) + \ld_{\Rats(\pow)}(\ktup)
\end{equation}
and
\begin{align}
 \ld_\Rats(\pow \xtup,\ktup/\xtup,\pow^{-1}\ktup) & \le  \ld_\Rats(\pow \xtup,\ktup/\xtup) \notag\\
 & =  \ld_\Rats(\pow \xtup/\ktup,\xtup) + \ld_\Rats(\ktup/\xtup) \notag\\
 & \le  \ld_\Rats(\pow \xtup/\ktup,\xtup) + \ld_\Rats(\ktup) \notag\\
 & =  \ld_\Rats(\pow \xtup/\ktup,\xtup) + \ld_{\Rats(\pow)}(\ktup)  \label{e9}
 \end{align}
the last line holding by lemma~\ref{lemma:LDElem}(ii), since $\Rats(\pow) \linearlydisjoint_\Rats C$ and $\ktup \subseteq \ker \subseteq C$.

Putting together (\ref{e7}), (\ref{e8}), and (\ref{e9}) gives 
$\ld_\Rats(\pow\xtup/\xtup,\ker) \ge \ld_{\Rats(\pow)}(\xtup/\ker)$
as required.
\end{proof}


\end{document}